\documentclass[11pt,a4paper]{article}
\usepackage{graphicx}
\usepackage{subfigure}
\usepackage{url}
\usepackage{color}
\usepackage{amsmath}
\usepackage{amsthm}
\usepackage{tikz}
\usepackage{epsfig}

\usepackage{amsfonts}

\newtheorem{theorem}{Theorem}

\newtheorem{lemma}[theorem]{Lemma}

\title{Reversible perturbations of conservative H\'enon-like maps}

\author{
	M.\,S.\,Gonchenko$^1$, S.\,V.\,Gonchenko$^2$, %A.\,O.\,Kazakov$^2,3$,
	K.\,Safonov$^3$ %, D.\,V.\,Turaev$^3$		
	\\%[12pt]
	{\small
		$^1$ Universitat Polit\`ecnica de Catalunya, Barcelona, Spain}\\
	{\small		$^2$ %Scientific and Educational
Mathematical Center ``Mathematics of Future Technologies'',}\\
{\small Lobachevsky State University of Nizhny Novgorod, Nizhny Novgorod, Russia}\\
	{\small		$^3$ National Research University Higher School of Economics,} \\
{\small 25/12 Bolshaya Pecherskaya Ulitsa, 603155 Nizhny Novgorod, Russia}\\
{\small { \texttt{marina.gonchenko@upc.edu}, 	\texttt{sergey.gonchenko@mail.ru}, 			 \texttt{safonov.klim@yandex.ru}	
} }
}

\date{}

\begin{document}
\maketitle

\begin{abstract}
For area-preserving H\'enon-like maps and their compositions, we consider smooth perturbations
that keep the reversibility of the initial maps but destroy their conservativity. For constructing such perturbations, we use two methods, the original method based on reversible properties of maps written in the so-called cross-form, and the classical Quispel-Roberts method based on a variation of involutions of the initial map. We study symmetry breaking bifurcations of symmetric periodic points in reversible families containing quadratic conservative orientable and nonorientable H\'enon maps as well as the product of two asymmetric H\'enon maps (with the Jacobians $b$ and $b^{-1}$).
\end{abstract}

\section{Introduction}

Among dynamical systems of various classes, the so-called {\em reversible systems}, which are characterized by invariance with respect to time reversal, are of special interest that can be explained by two main circumstances: first, such systems often appear in applications \cite{LR98}, and second, they form a class of systems with special symmetries, and therefore require the development of very specific mathematical methods for their study \cite{D76,RQ92}. To date, much is already known about the dynamics of reversible systems and the list of related papers is very vast, the reader can find it, for example, in the review paper \cite{LR98}. Especially a lot of fundamental results were obtained in the case of two-dimensional reversible maps, which are the main object of our article, see e.g. \cite{RQ92}--\cite{LT12}.

Recall that a $C^r$-map (diffeomorphism) $f$ is said to be \emph{reversible} if it is conjugate to its inverse $f^{-1}$ by an involution ${\cal G}$, i.e. the following condition holds
$f = {\cal G} \circ f^{-1} \circ {\cal G}$, where ${\cal G}^2=Id$ and the diffeomorphism ${\cal G}$ is also at least $C^r$-smooth.

Recently, the study of dynamics of reversible systems has got a new motivation
due to the discovery of the new, third, form of dynamical chaos,
the so-called {\em mixed dynamics} \cite{Gon16,GT17}, which
is characterized by the
the principal inseparability of dissipative elements of dynamics (attractors and repellers)
from conservative ones.
The above property makes the mixed dynamics fundamentally different from the two other classical forms of dynamical chaos, the conservative and dissipative chaos.

The most known type of conservative dynamics is demonstrated by Hamiltonian systems
or, more generally, systems preserving the phase volume. From the point of view of topological
dynamics, the conservative dynamics is characterized by the fact that the entire phase space of the
corresponding system is chain transitive, i.e., any two points can be connected by $\varepsilon$-orbits
for any $\varepsilon >0 $.  Recall that a sequence of points $x_1,...,x_n$ is called an $\varepsilon$-orbit (of length $n$) for a map $f: y_{i+1} = f(y_i)$
if $\mbox{{\rm dist}}\left(f(x_j), x_{j+1}\right) < \varepsilon$ for all $j=1,...,n-1$. We will say that an $\varepsilon$-orbit $x_1,...,x_n$ connects the
points $x_1$ and $x_n$.

The dissipative dynamics has a completely different nature: it is associated with the existence of
``holes'' (absorbing and repelling domains) in the phase space $M$. Recall that an open domain $D$
is said to be absorbing (repelling) if its image under the action of a map $T$ (a map $T^{-1}$)
lies strictly inside it.
By definition, a dissipative attractor, closed stable invariant set,  resides in some absorbing domain $D_a$,
analogously, a (dissipative) repeller resides in some repelling domain $D_r$.
Accordingly, we have $D_a\cap D_r = \emptyset$ here.

As for the mixed dynamics, unlike the conservative case, the phase space is not chain transitive, since infinitely many  dissipative attractors and repellers exist here,  which  intersect along closed invariant sets, the so-called reversible cores, having neutral (conservative-like) type of stability.
The latter means that
the reversible core itself attracts nothing and repels nothing -- for any nearby point, its forward orbit tends to the nearest attractor and backward orbit tends to the nearest repeller.\footnote{Note that the reversible core can be very vast and occupy very big part of the phase space that certain examples  show, see e.g. \cite{GGK13,GGKT17}.}
In other words, here, unlike the dissipative case, it is
impossible to construct a set of disjoint absorbing and repelling domains. The explanation of this phenomenon
from the topological point of view
was given in \cite{GT17} based on the concept of attractor going back to D. Ruelle \cite{R81}.

Note that one of the main fundamental properties of systems with mixed dynamics, which
can also be considered as a criterion and even as its definition (from the mathematical point of view), is the existence of so-called {\em absolute Newhouse regions} \cite{GST97,T11,T15}.
Recall that Newhouse regions are open regions in
the space of dynamical systems (or in the parameter space) in which systems with homoclinic
tangencies are dense (or values of parameters corresponding to systems with homoclinic tangencies
are dense) \cite{N79,GST93b,PV94,R95}.
It was shown by Newhouse himself \cite{N74} that, in the dissipative case,
there may exist Newhouse regions in which systems with infinitely many stable and saddle periodic
orbits are dense and, moreover, generic, i.e., they form subsets of the second Baire category (see also \cite{GST08}).
The absolute Newhouse regions are characterized
by the following property: systems with infinitely many periodic orbits of all possible types (sinks,
sources, and saddles)
are generic in such regions, and
these orbits are inseparable from each other, i.e., the closures of the sets of orbits of different types have nonempty intersections.

The absolute Newhouse regions exist for two-dimensional reversible  maps as well \cite{LS04,DGGLS13,GGS20}.
However, the dynamics of systems from reversible absolute Newhouse regions is much richer than for those in general case. In particular, as shown in \cite{LS04}, diffeomorphisms with
infinitely many coexisting periodic sinks, sources, and symmetric elliptic periodic orbits are generic in reversible absolute
Newhouse regions.

A periodic orbit of a reversible map $f$ is called {\em symmetric} if it is invariant
with respect to the involution ${\mathcal G}$, i.e., if its points are posed ${\mathcal G}$-symmetrically around the set $\mbox{{\rm Fix}}(h) = \{x: {\mathcal G}(x) = x\}$ of fixed points of the
involution, thus, if $Q$ is such an orbit, then $Q = {\mathcal G}(Q)$. In the two-dimensional case, a symmetric periodic orbit of an orientable reversible map $f$
has multipliers $\lambda$ and $\lambda^{-1}$. In general case $\lambda\neq \lambda^{-1}$, symmetric periodic points can be divided into two
types: saddle points, if $\lambda\neq \pm 1$ is real, and elliptic points, if $\lambda_{1,2} = e^{\pm i\varphi}$ and $0<\varphi<\pi$. The saddle
points are rough (structurally stable). As for elliptic points, although they are very similar to
conservative elliptic points \cite{Sevruyk}, they can differ greatly from the latter, as shown by the following results \cite{GT17,GLRT14}:

\begin{itemize}
\item
{\em all symmetric elliptic periodic orbits of a $C^r$-generic $(r = 1,...,\infty )$
two-dimensional reversible map are limits of periodic sinks and sources.}\footnote{The genericity is understood here in the sense that reversible maps with the indicated properties form a subset of the second Baire category in the space of $C^r$-smooth reversible maps having symmetric elliptic periodic orbits.}
\item
{\em all symmetric elliptic periodic orbits of a $C^r$-generic ($r = 1,...,\infty$) two-dimensional reversible
map are reversible cores;}
\item
{\em the generic elliptic periodic point of a reversible map is totally stable, i.e., stable under permanently acting perturbations (Lyapunov stable for $\varepsilon$-orbits), while in the case of area-preserving maps any such orbit is unstable (although it is Lyapunov orbitally stable).}
\end{itemize}

Thus, the reversible mixed dynamics manifests itself locally, but wherever symmetric elliptic points exist.
This important circumstance, of course, testifies to the fact that mixed dynamics should be viewed
as one of the fundamental properties of reversible systems. Moreover, the above result shows that symmetric (elliptic and other) orbits form a skeleton of global mixed dynamics, as they compose naturally that invariant set, reversible core, which simultaneously separates and connects the attractor and repeller.

The other important circumstance testifying to the universality of mixed dynamics in reversible
systems is what can be called the conjecture on Reversible Mixed Dynamics:

\begin{itemize}
\item
{\em Near any reversible map with a symmetric homoclinic tangency or a symmetric
nontransversal heteroclinic cycle, there are absolute reversible Newhouse regions.}
\end{itemize}

This RMD-conjecture was formulated in \cite{DGGLS13} and was almost immediately proved in \cite{GLRT14} for Newhouse
regions from the space of reversible systems in the $C^r$-topology with $2\leq r \leq\infty$. In the analytical case, as well as in the case of parameter
families, the RMD-conjecture was proved only for the so-called a priori non-conservative reversible diffeomorphisms \cite{LS04,GGS20,DGGL18},
when the (heteroclinic) cycle contains non-conservative elements (for example, saddles with the
Jacobians greater and less than one, or pairs of nonsymmetric homoclinic tangencies of a symmetric
saddle point, as in \cite{DGGL18}), as well as for reversible maps with symmetric heteroclinic cycles of conservative
type \cite{DGGLS13}.

Essentially, it remains to consider only two most interesting cases:
{\em reversible maps with symmetric quadratic and cubic homoclinic tangencies}.
However, these two cases are also the most difficult, since, in the principle plan, the main problem of this topic is connected
with the study of symmetry breaking bifurcations in first return maps constructing near orbits of symmetric homoclinic tangencies.
In the main order, these maps coincide with the conservative H\'enon-like maps: the standard H\'enon map $\bar x = y,\; \bar y = M - x - y^2$ (in the case of symmetric quadratic homoclinic tangency) and the reversible cubic H\'enon maps $\bar x = y,\; \bar y = -x + My \pm y^3$ (appearing near symmetric cubic homoclinic tangencies of two different types).

Both these maps are only certain truncated normal forms for the complete first return maps, and they demonstrate exclusively conservative dynamics.
What can be said about the dynamics of these maps under perturbations that keep the  reversibility, and how can dissipative dynamics elements appear here, such as periodic sinks, sources, or saddles with a Jacobian other than 1?
This is still open problem which requires solving the following issues.
\begin{itemize}
\item
How to construct perturbations of area-preserving H\'enon-like maps
which maintain their reversibility, but destroy the conservativity?
\item
What is the structure of symmetry breaking bifurcations under such perturbations?
\end{itemize}
We consider these questions to be relevant and interesting
not only because their solution will give an addition to the theory of H\'enon-like maps \cite{Henon76,S79,Bir87,DM00,GGO17}, but also it will make a certain contribution to the theory of mixed dynamics of reversible systems.

In the current paper we deal with these questions. Accordingly, the paper is divided into two parts. In the first one, Sections~\ref{sec:revpert} and~\ref{sec:QR}, we consider two types of methods for the construction of reversible perturbations for conservative H\'enon-like maps.
The first method looks to be new: we call it ``cross-form perturbations'', see Section~\ref{sec:crossform}.  We apply this method for the conservative H\'enon-like maps
(\ref{eq:HenonMap}), see Section~\ref{sec:crossformH1}, and for compositions of two H\'enon-like maps, see Sections~\ref{sec:crossformH-2} and~\ref{sec:crossform2H}.
The second method is the classical method proposed in the paper \cite{RQ92} by Quispel and Roberts. We apply this Quispel-Roberts method  for map (\ref{eq:HenonMap}) in Section~\ref{sec:QR} and for the nonorientable conservative H\'enon-like map of the form $\bar x = - y, \bar y = -x + F(y)$ in Section~\ref{sec:asymH-}.\footnote{Note that,
formally, the cross-form perturbations method and Quispel-Roberts method give different results, at first sight. Of course, the Quispel-Roberts  method is more general, since it can be apply to any reversible maps, however, it is not very clear how certain perturbations can be obtained by means of it, in particular, those that the cross-form method gives.}

In the second part of the paper, Section~\ref{sec:sbrbif}, we study symmetry breaking bifurcations in one-parameter families of reversible non-conservative H\'enon-like maps, using those perturbations that were constructed in the first part of the paper. We show that the simplest bifurcations of this type are reversible pitchfork bifurcations of periodic orbits. We consider such families in the cases of the product of two (quadratic) H\'enon maps (Section~\ref{sec:sbrbif_dggls13}), the nonorientable conservative H\'enon map (Section~\ref{sec:bifnor}) and the orientable conservative H\'enon map (Sections~\ref{sec:sborH} and~\ref{sec:FO6}). In the first two cases we show that even symmetric fixed points can undergo pitchfork bifurcations and recover their structure. It is interesting that, in the case of orientable conservative H\'enon map, this bifurcation occurs starting only with an orbit of period 6 (no such bifurcation takes place for orbits of less period), that is very surprising.

\section{On construction of reversible perturbations for H\'enon-like maps and their compositions} \label{sec:revpert}

The conservative H\'enon-like maps are the two-dimensional planar diffeomorphisms that can be represented in the form
\begin{equation}
H: \;
%\begin{cases}
\bar x = y,\;\;\;
\bar y = - x + F(y),
%\end{cases}
\label{eq:HenonMap}
\end{equation}
where $F(y)$ is some nonlinear function (e.g. a polynomial).
Map (\ref{eq:HenonMap}) is area-preserving, with the Jacobian equal to 1, and reversible with respect to the linear involution  $h: (x, y) \to (y, x)$.
Indeed, $H^{-1}$ takes the form $x=\bar y, y = - \bar x + F(\bar y)$; the relation $h \circ H^{-1} \circ h$ means, due to the simplicity of $h$, that we need to make interchanges
$x \leftrightarrow y, \bar x \leftrightarrow \bar y$ in the formula for $H^{-1}$, after which we get (\ref{eq:HenonMap}).

In this section we consider two methods for the construction of such sufficiently smooth (analytic) perturbations of H\'enon-like maps (\ref{eq:HenonMap}) and their compositions that destroy the conservativity of these maps but keep their reversibility with respect to the involution $h$.

\subsection{Cross-form perturbations} \label{sec:crossform}

The first method to obtain reversible perturbations is based on the following cross-form map
\begin{equation}
%\begin{cases}
g: (x,y)\to (\bar x, \bar y)\;: \;\;\; \bar x = G(x, \bar y),\;\;\;
y = G(\bar y, x).
%\end{cases}
\label{eq:RevMap}
\end{equation}
Note that the map~(\ref{eq:RevMap}) is reversible with respect to the involution $h: (x,y)\to (y,x)$.
The proof is immediate: the map $g^{-1}$ has the form $x = G(\bar x, y),\;\;\; \bar y = G(y, \bar x)$, and the composition $h\circ g^{-1}\circ h$ means that we need to make interchanges
$x \leftrightarrow y$ and $\bar x \leftrightarrow \bar y$ in $g^{-1}$, which leads to (\ref{eq:RevMap}).

We introduce certain notations for the derivatives of functions:
\begin{itemize}
\item
$F^\prime(\rho)$ denotes the first derivative of the function $F(y)$ at the point $y=\rho$;
\item
for a smooth function $s(x,y)$, we denote
$$
u(x,y) = \frac{\partial s(x,y)}{\partial x}, \;\; v(x,y) = \frac{\partial s(x,y)}{\partial y}.
$$
and
$$
s_x(\xi,\eta) = u(\xi,\eta),\; s_y(\xi,\eta) = v(\xi,\eta).
$$
Thus, the subscripts  $x$ and $y$ means the differentiation with respect to the first and second variables, respectively.
\end{itemize}

\begin{lemma} \label{lm:crosjac}
The Jacobian of map (\ref{eq:RevMap})  takes the form
\begin{equation}
J \; = \; \frac{G_x(x, \bar y)}{G_x(\bar y, x)}.
\label{eq:JRevMap}
\end{equation}
\end{lemma}

\begin{proof}
It  follows from (\ref{eq:RevMap}) that
$$
%\begin{array}{l}
\frac{\partial\bar x}{\partial x} = G_x(x,\bar y) + G_y(x,\bar y)\frac{\partial\bar y}{\partial x}, \;\;\;
\frac{\partial\bar x}{\partial y} = G_y(x,\bar y) \frac{\partial\bar y}{\partial y},
$$
$$
 0 = G_x(\bar y,x) \frac{\partial\bar y}{\partial x} + G_y(\bar y,x), \;\;\; 1 =  G_x(\bar y,x) \frac{\partial\bar y}{\partial y}
%\end{array}
$$
Then we get
$$
\frac{\partial\bar y}{\partial y}  = \frac{1}{G_x(\bar y,x)}, \; \frac{\partial\bar y}{\partial x}  = - \frac{G_y(\bar y,x)}{G_x(\bar y,x)},\;
 \frac{\partial\bar x}{\partial x}  = G_x(x,\bar y) - \frac{G_y(x,\bar y)G_y(\bar y,x)}{G_x(\bar y,x)}, \; \frac{\partial\bar x}{\partial y}  = \frac{G_y(x,\bar y)}{G_x(\bar y,x)},
$$
and, as a result, we deduce formula (\ref{eq:JRevMap}) for the Jacobian
$
J = \partial\bar x /\partial x \cdot\partial\bar y/\partial y \; - \;\partial\bar x/\partial y \cdot \partial\bar y/\partial x .
$
\end{proof} 

Therefore, once having a conservative map written in the implicit form~(\ref{eq:RevMap}), we can simply add a perturbation in such a way that the cross-form is preserved, and the perturbed system will be reversible.

\subsection{Cross-form perturbation of~(\ref{eq:HenonMap})} \label{sec:crossformH1}

The idea to write a perturbation for the map~$H$, given in~~(\ref{eq:HenonMap}), comes from the formal solution of the second equation of~(\ref{eq:HenonMap}) for~$y$:
$
y=F^{-1}(x+\bar y)
$. Then the map~$H$ is rewritten in the cross-form
$$
H: \;
\bar x = F^{-1}(x+\bar y),\;\;\;
y = F^{-1}(\bar y+x).
$$
Thus, the perturbation of the form
$$
\tilde H: \;
\bar x = F^{-1}(x+\bar y) + \varepsilon(x,\bar y),\;\;\;
y = F^{-1}(\bar y+x) + \varepsilon(\bar y, x)
$$
is formally reversible. For this map we obtain from the second equation that $F^{-1}(\bar y+x) = y - \varepsilon(\bar y, x)$ and $\bar y + x = F(y- \varepsilon(\bar y, x))$. Then  map~$\tilde H$ takes the following form
\begin{equation}
\tilde H: \;
%\begin{cases}
\bar x = y + \varepsilon(x, \bar y) - \varepsilon(\bar y, x),\;\;\;
\bar y = - x + F(y-\varepsilon(\bar y, x)).
%\end{cases}
\label{eq:HenonMapPert}
\end{equation}
By construction, map (\ref{eq:HenonMapPert}) should be reversible, however, the operator $F^{-1}$ is only formal, therefore the reversibility of $\tilde H$ must be proved directly. This is done in the following lemma.

\begin{lemma}
		The map $\tilde H$, defined in~(\ref{eq:HenonMapPert}), is reversible with respect to the involution $h: (x,y)\to (y,x)$.	
\end{lemma}

\begin{proof} To prove the reversibility of $\tilde{H}$, we have to show that $\tilde{H}=h\circ \tilde{H}^{-1}\circ h$.
	
The inverse map $\tilde{H}^{-1}$ is obtained after swapping the bar and no-bar variables $\bar x \leftrightarrow x, \bar y \leftrightarrow y$, i.e.,
\begin{equation}\label{eq:tldHm1}
	\tilde{H}^{-1}\;: \;\;\;  \bar x = - y + F(\bar y-\varepsilon(y, \bar x)), \;\;\; \bar y =   x - \varepsilon(\bar x, y) + \varepsilon(y, \bar x).
\end{equation}
After exchanging $x \leftrightarrow y$ and $\bar x \leftrightarrow \bar y$ in (\ref{eq:tldHm1}),  due to the involution $h$, we get the expression for $h\circ \tilde{H}^{-1}\circ h$ which coincides with~(\ref{eq:HenonMapPert}).
\end{proof}

\begin{lemma}\label{lm:J_Htld}
The Jacobian of the map~(\ref{eq:HenonMapPert}) takes the following formula
   \begin{equation}\label{eq:J_Htld}
   J=\frac{1+ F'(y-\varepsilon(\bar y,x)) \varepsilon_x(x,\bar y)}{1+F'(y-\varepsilon(\bar y, x)) \varepsilon_x(\bar y, x)}.
   \end{equation}
\end{lemma}

\begin{proof}
Differentiating the first equation of~(\ref{eq:HenonMapPert}) with respect to $x$ and $y$ we get
$$
\frac{\partial \bar x}{\partial x} = \varepsilon_x(x,\bar y) +\varepsilon_y(x, \bar y) \frac{\partial \bar y}{\partial x} - \varepsilon_x (\bar y, x) \frac{\partial \bar y}{\partial x} - \varepsilon_y (\bar y, x), \;\;\; \frac{\partial \bar x}{\partial y} = 1 + \varepsilon_y(x,\bar y) \frac{\partial \bar y}{\partial y} - \varepsilon_x(\bar y, x) \frac{\partial \bar y}{\partial y}.
$$
Therefore, we have
\begin{equation}\label{eq:Jl3}
J= \frac{\partial \bar x}{\partial x} \frac{\partial \bar y}{\partial y} -\frac{\partial \bar x}{\partial y} \frac{\partial \bar y}{\partial x} = \left( \varepsilon_x(x,\bar y) - \varepsilon_y(\bar y, x)\right) \frac{\partial \bar y}{\partial y} - \frac{\partial \bar y}{\partial x}.
\end{equation}
We find the derivatives ${\partial \bar y}/{\partial x}$ and ${\partial \bar y}/{\partial y}$ from the second equation of~(\ref{eq:HenonMapPert}) by its implicit differentiation
$$
\frac{\partial \bar y}{\partial x} = \frac{-1-F'(y-\varepsilon(\bar y, x)) \varepsilon_y(\bar y, x)}{1+F'(y-\varepsilon(\bar y, x)) \varepsilon_x(\bar y, x)}, \;\;\;
\frac{\partial \bar y}{\partial y} = \frac{F'(y-\varepsilon(\bar y,x))}{1+F'(y-\varepsilon(\bar y, x)) \varepsilon_x(\bar y, x)}.
$$
After substituting these into~(\ref{eq:Jl3}), we obtain~(\ref{eq:J_Htld}).
\end{proof}

It is worth mentioning that if the perturbation $\varepsilon(x,y)$ in~(\ref{eq:HenonMapPert}) is a symmetric function, i.e. $\varepsilon(x,y)=\varepsilon(y,x)$, the perturbed map~(\ref{eq:HenonMapPert}) takes the simpler form
\begin{equation}\label{eq:HMPertSymm}
\tilde H: \;
%\begin{cases}
\bar x = y,\;\;\;
\bar y = - x + F(y-\varepsilon(x, \bar y)),
\end{equation}
and the same formula~(\ref{eq:J_Htld}) holds for the Jacobian.

Note that the perturbed systems~(\ref{eq:HenonMapPert}) and (\ref{eq:HMPertSymm}) contain perturbing terms inside a nonlinear function $F$, and, hence, it is hard to iterate the maps --
one needs to solve the second equations for $\bar y$ and calculate $\bar y = f(x,y)$. In the following subsections we show that using cross-form~(\ref{eq:RevMap}) it is possible to construct reversibility preserving perturbations of another kind which allows to iterate the maps directly. We also show that such perturbations can be constructed by the Quispel-Roberts method, see Section~\ref{sec:QR}.

\subsection{Perturbations of $H^{-2}$}  \label{sec:crossformH-2}

The cross-form reversible perturbations can be easily constructed for the map $H^{-2}$ that is the square of $H^{-1}$, i.e., the inverse map to the conservative H\'enon-like map $H$.
We obtain from (\ref{eq:HenonMap}) that map $H^{-1}$ takes the form
$$
  H^{-1}: \;\;\; \bar x = -y + F(x), \;\;\; \bar y =x.
$$
Then the map $H^{-2}$ is written as
\begin{equation}\label{eq:Hm-2}
  H^{-2}=H^{-1}\circ H^{-1}: \;\;\; \bar x = -x + F(-y + F(x)), \;\;\; \bar y = - y + F(x),
\end{equation}

\begin{lemma}
  The map of the form
  \begin{equation}\label{eq:tldHm2}
  \tilde{H}^{-2}: \;\;\; \bar x = -x + F(\bar y) + \varepsilon(x,\bar y), \;\;\; \bar y= -y + F(x) + \varepsilon(\bar y, x)
  \end{equation}
  is reversible  with respect to the involution $h:(x,y)\to(y,x)$. The Jacobian of $\tilde{H}^{-2}$ is
  $$
  J=\frac{1-\varepsilon_x(x,\bar y)}{1-\varepsilon_x(\bar y, x)}.
  $$
\end{lemma}

\begin{proof} Map (\ref{eq:Hm-2}) can be presented in the cross-form as follows
  $$
  H^{-2}: \;\;\; \bar x = -x + F(\bar y), \;\;\; y = - \bar y + F(x),
  $$
that have form (\ref{eq:RevMap}) with $G(x,y) = - x + F(y)$. Thus, the perturbation
$\bar x = -x + F(\bar y) + \varepsilon(x,\bar y), \;\; y = - \bar y + F(x) + \varepsilon(\bar x,y)$ is  what we need, and it takes the form~(\ref{eq:tldHm2}). The desired formula for the Jacobian $J(\tilde H^{-2})$ is obtained from (\ref{eq:JRevMap}) for $G(x,y) = - x + F(y) + \varepsilon(x,y)$.
\end{proof}

The form of the map~(\ref{eq:tldHm2}) allows to write the map explicitly for some perturbations. For example, if $\varepsilon(x,y)$ is linear in $x$, i.e. $\varepsilon(x,y)=\alpha(y)+x \beta(y)$, then the map yields
$$
\tilde{H}^{-2}: \;\;\; \bar x = -x + F(\bar y) + \alpha(\bar y) + x \beta(\bar y), \;\;\; \bar y= \frac{-y + F(x) + \alpha(x) }{1- \beta(x)}
$$
and its Jacobian is
$$
J= \frac{1- \beta(\bar y)}{ 1- \beta(x)}.
$$
Hence, the new map is a diffeomorphism in some ball $\{x\in \mathbb{R} : \|(x,y)\|\leq R_\beta\}$, where $R_\beta\to\infty$ as $|\beta|\to 0$.
Besides, for some special functions $\beta(x)$ (for instance, $\beta(x)=\mu \arctan(x)$ with sufficiently small $\mu$), the map is an analytical diffeomorphism in the whole plane $\mathbb{R}^2$.

\subsection{Perturbations of conservative compositions of two non-conservative H\'enon-ilke maps} \label{sec:crossform2H}

The next approach is connected with perturbations of the product of  two asymmetric non-conservative H\'enon-like maps $H_1$ and $H_2$ of the form
\begin{equation}
H_1: \;\;\; \bar x = y, \;\;\; \bar y = b x + F(y)\;\; \mbox{{\rm and}}\;\;  H_2: \;\;\; \bar x = y, \;\;\; \bar y = \frac{1}{b} x - \frac{1}{b} F(y).
\label{eq:Hlms}
\end{equation}
These maps have the Jacobians $b$ and $1/b$, respectively, and their nonlinearities are asymmetric.
Their inverse maps are
$$
H_1^{-1}: \;\;\; \bar x = \frac{1}{b} y - \frac{1}{b} F(x), \;\;\; \bar y = x \;\; \mbox{{\rm and}}\;\; H_2^{-1}: \;\;\; \bar x = b y + F(x), \;\;\; \bar y =  x,
$$
respectively. The composition $ H_{12}^{-1} = H_1^{-1}\circ H_2^{-1}$ of the last two maps can be written as
$$
H_{12}^{-1}: \;\;\; \bar x = \frac{1}{b} x - \frac{1}{b} F\left( b y +  F(x)\right), \;\; \bar y = b y +  F(x),
$$
or, in the cross-form, as
\begin{equation}
H_{12}^{-1}: \;\;\; \bar x = \frac{1}{b} x - \frac{1}{b} F\left(\bar y \right), \;\; y = \frac{1}{b} \bar y  - \frac{1}{b}  F(x).
\label{eq:H12-1}
\end{equation}
Thus, the composition $H_1^{-1}\circ H_2^{-1}$ has the cross-form ~(\ref{eq:RevMap}) with $G(x,y)= b^{-1}(x - F(y))$. This implies the following result

\begin{lemma}\label{lm:H12tld}
The map
\begin{equation}
\tilde{H}_{12}^{-1}: \;\;\; \bar x = \frac{1}{b} x  - \frac{1}{b} F(\bar y) + \varepsilon(x,\bar y), \;\;\;  y = \frac{1}{b} \bar y - \frac{1}{b} F(x) + \varepsilon(\bar y, x)
\label{eq:H12-1e}
\end{equation}
is a reversible perturbation of $H_1^{-1}\circ H_2^{-1}$ that keeps the involution $h:(x,y)\to (y,x)$, and
\begin{equation} \displaystyle
J\left(\tilde{H}_{12}^{-1}\right)=\frac{1+b\varepsilon_x(x,\bar y)}{1+b\varepsilon_x(\bar y, x)}.
\label{eq:J(H12)}
\end{equation}

\end{lemma}

\section{Quispel-Roberts method for construction of reversible perturbations.} \label{sec:QR}

The basic elements of the theory of reversible systems were developed in the famous paper \cite{RQ92} by Quispel and Roberts. In particular, in this paper general methods for the construction of reversible perturbations of reversible maps were proposed. One of such methods is based on the following two facts:
\begin{itemize}
\item[1)]  Any reversible map can be represented as a composition of its two involutions.
\item[2)]  If $g$ is an involution, then $\tilde g = T^{-1}\circ g \circ T$ is also involution, if the map $T$ is a diffeomorphism.
\end{itemize}

Indeed, for item 1), if $g$ is an involution of a map $f$,
we have $f = g\circ f^{-1}\circ g = g\circ \left(f^{-1}\circ g\right)$ and the map  $f^{-1}\circ g$ is also involution, since
$$
\left(f^{-1}\circ g\right)^2 = f^{-1}\circ g \circ f^{-1}\circ g = f^{-1}\circ \left( g \circ f^{-1}\circ g \right) =  f^{-1}\circ f = \;\mbox{{\rm id}}.
$$
For item 2), we obtain
$\tilde g^2 = T^{-1}\circ g \circ \left( T \circ T^{-1} \right) \circ g \circ T = T^{-1}\circ \left( g\circ g \right) \circ T = T^{-1}\circ T = \;\mbox{{\rm id}}. $

The conservative H\'enon-like map $H$, given by (\ref{eq:HenonMap}), can be also presented as the product $H = h_1\circ h_2$ of two involutions:
\begin{equation}
h_1 = h = \left\{\begin{array}{l} \bar x = y, \\ \bar y = x \end{array}\right. \;\;\mbox{{\rm and}}\;\;
h_2 = \left\{\begin{array}{l} \bar x = - x + F(y), \\ \bar y = y \end{array}\right.
\label{eq:2invH}
\end{equation}

Thus, we can construct reversible perturbations of $H$ by means of changing their involutions. For our goals,
we keep the involution $h_1 = h$ and take new involution $\tilde h_2$  as the perturbation $\tilde h_2 = T^{-1}\circ h_2 \circ T$ of the involution $h_2$ by means of a map
$T$ that is close to the identity map $\bar x = x, \bar y = y$. The following lemma summarizes results of the corresponding calculations.

\begin{lemma}\label{lm:til_H}
The map
\begin{equation}
\hat H : \left\{
\begin{array}{l}
\bar x = y + \varepsilon_2(x,y) - \varepsilon_2(\bar y,\bar x), \\
\bar y = -x  + F \left(y + \varepsilon_2(x,y)\right) - \varepsilon_1(x,y) -  \varepsilon_1(\bar y,\bar x)
\end{array}
\right.
\label{eq:tildeH}
\end{equation}
is a reversible perturbation of the conservative H\'enon-like map $H$, given in (\ref{eq:HenonMap}), that is constructed in the form $\hat H = h_1\circ \tilde h_2$, where $\tilde h_2 = T^{-1}\circ h_2 \circ T$ and the map
$T: \; \bar x = x + \varepsilon_1 (x,y), \bar y = y + \varepsilon_2 (x,y)$ is assumed to a near identity map.
\end{lemma}

\begin{proof}
We will find first the new involution $\tilde h_2 = T^{-1}\circ h_2 \circ T$.
By~(\ref{eq:2invH}), the composition  $h_2 \circ T : (x,y)\to (x^\prime,y^\prime)$ can be written as
$$
h_2 \circ T : \left\{\begin{array}{l} x^\prime = -x - \varepsilon_1(x,y) + F(y + \varepsilon_2(x,y)), \\ y^\prime = y + \varepsilon_2(x,y). \end{array}\right.
$$
We can write the map $T^{-1}: (x^\prime,y^\prime)\to (\bar x,\bar y)$ as follows
$\bar x + \varepsilon_1 (\bar x,\bar y) = x^\prime, \; \bar y + \varepsilon_2 (\bar x,\bar y)=y^\prime$.
Then for the new involution $\tilde h_2$, we get the following expression
$$
\tilde h_2 = T^{-1}\circ h_2 \circ T : \left\{\begin{array}{l} \bar x +  \varepsilon_1(\bar x, \bar y) = -x - \varepsilon_1(x,y) + F(y + \varepsilon_2(x,y)), \\
\bar y  + \varepsilon_2(\bar x, \bar y) = y + \varepsilon_2(x,y). \end{array}\right.
$$
After this, formula (\ref{eq:tildeH}) for the map $\hat H = h_1\circ \tilde h_2$ is easy obtained: we only need  to replace $\bar x \leftrightarrow \bar y$ in this expression for $\tilde h_2$ ($x$ and $y$ are not changed).
\end{proof}

\begin{lemma}\label{lm:Jtil_H}
	The Jacobian of the perturbed map $\hat H$ is
\begin{equation} \label{eq:QRJac}
J(\hat {H})=\frac {  \left( 1 +\varepsilon_{2y}(x,y)\right) \left( 1 +\varepsilon_{1x}(x,y)\right) - \varepsilon_{2x}(x,y)\varepsilon_{1y}(x,y)}
{\left(1 +\varepsilon_{2y}(\bar y,\bar x)\right) \left( 1 +\varepsilon_{1x}(\bar y,\bar x)\right) - \varepsilon_{2x}(\bar y,\bar x)\varepsilon_{1y}(\bar y,\bar x)}.
\end{equation}
\end{lemma}

\begin{proof}
We calculate the derivatives ${\partial \bar x}/{\partial x}$, ${\partial \bar y}/{\partial x}$, ${\partial \bar x}/{\partial y}$ and ${\partial \bar y}/{\partial y}$ from (\ref{eq:tildeH}):
$$
\begin{array}{l}
\displaystyle \left(1 + \varepsilon_{2y} (\bar y, \bar x) \right) \frac{\partial \bar x}{\partial x} =  \varepsilon_{2x} (x,y)  -  \varepsilon_{2x}(\bar y,\bar x) \frac{\partial \bar y}{\partial x}, \;\;\;
\left(1 + \varepsilon_{2y} (\bar y, \bar x) \right) \frac{\partial \bar x}{\partial y} = 1 +  \varepsilon_{2y} (x,y)  -
\varepsilon_{2x}(\bar y,\bar x) \frac{\partial \bar y}{\partial y}, \\ \\
\displaystyle \left(1 + \varepsilon_{1x} (\bar y, \bar x) \right) \frac{\partial \bar y}{\partial x} = - 1 + F^\prime \left(y + \varepsilon_2(x,y)\right)\cdot \varepsilon_{2x} (x,y)
 - \varepsilon_{1x} (x,y) - \varepsilon_{1y}(\bar y,\bar x) \frac{\partial \bar x}{\partial x}, \\ \\
\displaystyle \left(1 + \varepsilon_{1x} (\bar y, \bar x) \right) \frac{\partial \bar y}{\partial y} = F^\prime \left(y + \varepsilon_2(x,y)\right)\cdot (1+\varepsilon_{2y} (x,y))  -
\varepsilon_{1y} (x,y) - \varepsilon_{1y}(\bar y,\bar x) \frac{\partial \bar x}{\partial y}.
\end{array}
$$
Solving this system for the partial derivatives, we get the formula (\ref{eq:QRJac}) for the Jacobian.
\end{proof}

Notice that among the perturbations in the form~(\ref{eq:tildeH}) we can select more simple ones which preserve reversibility and destroy conservativity. Let us consider two examples.

\textbf{Example 1.} We consider the case with $\varepsilon_2\equiv 0$. Then the map~(\ref{eq:tildeH}) takes the form
\begin{equation}
\hat H : %\left\{
\begin{array}{l}
\bar x = y, \;\; \bar y = -x  + F(y) - \varepsilon_1(x,y) -  \varepsilon_1(\bar y,\bar x)
\end{array}
%\right.
\label{eq:tildeHe1}
\end{equation}
and the Jacobian of this map
\begin{equation}\label{eq:J_eps3y}
J = \frac{1 + \varepsilon_{1x}(x,y)}{1 + \varepsilon_{1x}(\bar y,\bar x)}
\end{equation}
is not 1 generally. Moreover, if, for example, $\varepsilon_{1}(x,y) = a_{20} x^2 + a_{11} xy + a_{02}y^2$, then (since $\bar x = y$)
$$
J = \frac{1 +  a_{11} y + 2 a_{20} x}{1 +  a_{11} \bar x + 2 a_{20} \bar y} = \frac{1 +  a_{11} y + 2 a_{20} x}{1 +  a_{11} y + 2 a_{20} \bar y},
$$
i.e. including quadratic terms $x y$ and $x^2$ into the perturbation $\varepsilon_1$ makes the Jacobian non-constant.

Other particular case of the function $\varepsilon_1(x,y)$ includes, for example,  $\varepsilon_1(x,y) =  x f_1(y) + f_2(y)$, where $f_1(0)=0, f_2(0)=f^\prime_2(0)=0$, i.e. $\varepsilon_1(x,y)$ being linear in $x$. Then $J = (1 + f_1(y))(1 + f_1(\bar x))^{-1}  \equiv 1$.

Let us consider a perturbation with $\varepsilon_1(x,y) = p(x) + q(y)$ and $p'(x) = v(x)$. Then
$\varepsilon_{1x}(x,y) = v(x)$, $\varepsilon_{1x}(\bar y,\bar x)  = v(\bar y)$  and, by (\ref{eq:J_eps3y}),
$$
J = \frac{1+v(x)}{1+v(\bar y)}
$$
Formally, it means that $J\not\equiv 1$. However, for any periodic orbit, the Jacobian $J_n$ of its first return map will be equal to 1.
Indeed, let $M_i(x_i,y_i)$, $i=1,...,n$, be the points of an $n$-periodic orbit $P$. Then, since $x_i = y_{i-1}$, we obtain that
\begin{equation} \label{eq:ruleJ}
J_n = \left. J(\hat{H}^n) \right|_{M_1} = \prod\limits_{i=1}^n  \frac{1+v(y_{i-1})}{1+v(y_{i+1})} \equiv 1
\end{equation}
since the nominator and denominator of this product contain the same factors. This means that any periodic orbit is conservative, any invariant sets with dense subsets of  periodic orbits  (for instance, horseshoes) are also conservative etc.

Moreover, we can claim that the dynamics of map $\hat H$ in the form (\ref{eq:tildeHe1}) with $\varepsilon_1(x,y) = p(x)+q(y)$ is totally conservative,
since this map possesses a smooth invariant measure.

Indeed, as known \cite{R78}, a measure $d\mu=\rho(x,y)dxdy$ is invariant if and only if  the density $\rho(x,y)$ is a fixed point of the Ruelle-Perron-Frobenius operator, i.e.
$$
\rho(x,y)=\frac{\rho \circ \hat{H}^{-1}(x,y)}{|J|}=\frac{1+v(\bar{y})}{1+v(x)} \cdot \rho \circ \hat{H}^{-1}(x,y).
$$
Let us  check that the function $\rho(x,y)=(1+v(y))\cdot (1+v(\bar{y}))$ satisfies this relation. For simplicity, take a point $(x_0,y_0)$ and denote its image by $(x_1,y_1)=\tilde{H}(x_0,y_0)$ and the preimage by $(x_{-1},y_{-1})=\tilde{H}^{-1}(x_0,y_0)$. Then we have
$$\rho(x_0,y_0)=(1+v(y_0))(1+v(y_1)), \;\;\; \rho(x_{-1},y_{-1})=(1+v(y_{-1}))(1+v(y_0)).$$
Since $x_0=y_{-1}$ we obtain
$$
\frac{1+v(y_1)}{1+v(x_0)} \cdot \rho(x_{-1},y_{-1})=\frac{1+v(y_1)}{1+v(x_0)} \cdot (1+v(x_0))(1+v(y_0))=\rho(x_0,y_0).
$$
Therefore, the measure
$$\mu(A)=\int_A (1+v(y))\cdot (1+v(\bar{y}))dxdy$$
is invariant for the map $\tilde{H}$.

\textbf{Example 2.}  Consider the case with $\varepsilon_1(x,y) \equiv 0$.
Then map~(\ref{eq:tildeH}) takes the form
\begin{equation}\label{eq:PertMarina1}
\hat{H}^{(2)}: \;\;\; \bar x   = y + \varepsilon_2(x,y) - \varepsilon_2(\bar y, \bar x), \;\;\; \bar y  = -x + F(y+\varepsilon_2(x,y)),
\end{equation}
and
$$
J(\hat{H}^{(2)}) = \frac {1 + \varepsilon_{2y}(x,y)}{1 + \varepsilon_{2y}(\bar y, \bar x)},
$$
i.e., $J$ is not 1 generally. However, not any perturbation $\varepsilon_2$ is suitable. For example, let the function $\varepsilon_{2y}(x,y)= v(x,y)$ be symmetric, i.e.
$v(x,y) = v(y,x)$ (as for $\varepsilon_2 = xy^2$). In this case, $J(\hat{H}^{(2)}|_{(x_i,y_i)}) = (1+v(x_i,y_i))(1+v(x_{i+1},y_{i+1}))^{-1}$, and when calculating the Jacobian of a periodic orbit $\{(x_i,y_i): i=1,\ldots,n, \hat{H}^{(2)}(x_i,y_i)= (x_{i+1}, y_{i+1}), x_{n}=x_1, y_{n}=y_1\}$ as in~(\ref{eq:ruleJ}), we get $J_n = 1$. At the same time, the perturbation  $\varepsilon_2 = \alpha xy$ is quite suitable. Indeed, the Jacobian
$J = (1+\alpha x)(1 + \alpha \bar y)^{-1}$ is not constant for $\alpha\neq 0$ and, moreover, since the function $\varepsilon_2(x,y)$ is linear in $y$, the map (\ref{eq:PertMarina1}) can be represented in the explicit form. Note that such perturbations were considered in \cite{GKSS19} while studying effects of reversible perturbations on 1:3 resonance in the conservative cubic H\'enon maps.

\subsection{Perturbations of nonorientable conservative H\'enon-like maps.} \label{sec:asymH-}

In this section we show that nonorientable conservative H\'enon-like maps also admit  reversible perturbations of the same types that have been considered in the previous sections for orientable maps.

We consider the following nonorientable conservative H\'enon-like map of the form
\begin{equation}
H_{-1}: \;\;\; \bar x = -y, \;\;  \bar y  = -x + F(y).
\label{eq:TH-0}
\end{equation}
It is easy to show that this map is reversible with respect to the involution $h: x\to y, y\to x$, if function $F(y)$ is even, i.e. $F(-y) = F(y)$ (in particular, it follows from Lemma~\ref{lm:Hninveps} below). By analogy with Lemma~\ref{lm:til_H}, we consider the following perturbation
\begin{equation}
\hat{H}_{-1}: \;\;\; \bar x = -y, \;\;  \bar y  = -x + F(y) -  \varepsilon(x,y) - \varepsilon(\bar y,\bar x),
\label{eq:TH-eps}
\end{equation}
where $\varepsilon(x,y)$ is some smooth function.

\begin{lemma}
If $F(y)$ is an even function, $F(y) =F(-y)$,
then the map $\hat{H}_{-1}$, given in~(\ref{eq:TH-eps}), is reversible with respect to the involution $h: x\to y, y\to x$, and
\begin{equation}
	J(\hat H_{-1}) = - \frac{1+\varepsilon_{x}(x,y)}{1+\varepsilon_{x}(\bar y,\bar x)}.
	\label{eq:J-eps0}
\end{equation}
	\label{lm:Hninveps}
\end{lemma}

\begin{proof} The inverse map $(\hat{H}_{-1})^{-1}$ takes the form
\begin{equation}\label{eq:TH_inv}
(\hat{H}_{-1})^{-1}: \;\;\; \bar x  = - y + F(-x) - \varepsilon(\bar x,\bar y) - \varepsilon(y,x), \;\;\; \bar y = - x.
\end{equation}
After the interchange $x\leftrightarrow y, \bar x \leftrightarrow \bar y$ in (\ref{eq:TH_inv}) we obtain the map $h \circ \hat{H}_{-1} \circ h$ that coincides with map (\ref{eq:TH-eps}), if $F(-y) = F(y)$, i.e., $F(y)$ is an even function.

For map $\hat H_{-1}$ in the form (\ref{eq:TH-eps}), we have that $J(\hat H_{-1}) = - \partial \bar x /\partial y \cdot \partial \bar y /\partial x$, since
$\partial \bar x /\partial x \equiv 0$. Then we have
$$
\frac{\partial \bar x}{\partial y} = -1, \;\; \frac{\partial \bar y}{\partial x} = -1 - \varepsilon_x(x,y) - \varepsilon_x(\bar y,\bar x) \frac{\partial \bar y}{\partial x}.
$$
This gives us the desired formula (\ref{eq:J-eps0}).
\end{proof}

\section{Symmetry breaking bifurcations in reversible perturbations of
H\'enon-like maps} \label{sec:sbrbif}

In this section we consider several examples of two-dimensional reversible maps that are perturbations of H\'enon-like maps and demonstrate reversible symmetry breaking bifurcations \cite{LT12} of fixed points or periodic orbits. Even with arbitrarily small perturbations,
such bifurcations lead to the appearance of dissipative elements of dynamics, although these bifurcations closely follow the corresponding bifurcations in the unperturbed area-preserving maps.
For example, a symmetric couple of elliptic or saddle orbits for the area-preserving map is transformed into  a symmetric couple containing sink and source or saddles with the Jacobians greater and less than 1 in the perturbed map, etc.
The knowledge of these bifurcations and conditions of their realization is very relevant to understand such phenomenon as the appearance of mixed dynamics at (reversible) perturbations of conservative systems \cite{Gon16,GT17,Tur15}. 

\subsection{Symmetry breaking bifurcations in the product of two quadratic H\'enon maps}\label{sec:sbrbif_dggls13}

Note that the product of two non-conservative asymmetric H\'enon maps  $H_1$ and $H_2$ of the form~(\ref{eq:Hlms}) with $F(y) = M - y^2$ appears naturally as the normal forms of first return maps near symmetric quadratic homoclinic or heteroclinic tangencies to symmetric periodic orbits of reversible diffeomorphisms~\cite{DGGLS13,DGGL18}. Accordingly, their local bifurcations under reversible perturbations can play a role of global symmetry breaking bifurcations leading to the onset of reversible mixed dynamics.

In this section we consider bifurcations of this type. They are bifurcations of fixed points in some one-parameter family of reversible maps that unfolds  the product of two non-conservative H\'enon maps  
$$ H_1: \bar x =y, \; \bar y = M - y^2 \text{ and } H_2: \bar x = y, \; b \bar y = M - y^2$$ 
with the Jacobians equal to $b$ and $b^{-1}$, respectively. Their compositions $H_2\circ H_1$ and $H_1^{-1}\circ H_2^{-1}$ are both area-preserving maps, and, moreover, the latter map  $T_2 = H_1^{-1}\circ H_2^{-1}$ can be written in the following cross-form (see formula (\ref{eq:H12-1e}))
$$
T_2 : \;\;\; \bar x = \frac{1}{b} x - \frac{M}{b} + \frac{1}{b}\bar y^2 , \;\;  y = \frac{1}{b} \bar y - \frac{M}{b} + \frac{1}{b} x^2.
$$

To study symmetry breaking bifurcations appearing at reversible perturbations of this map we embed it in the following one-parameter family
\begin{equation}
%\begin{array}{l}
{T}_{2\mu}: \;\;\;\;
\displaystyle \bar x  = - \frac{M}{b} + \frac{1}{b} x  + \frac{1}{b} \bar y^2 +\mu x \bar y, \;\;\;
\displaystyle y  = - \frac{M}{b} + \frac{1}{b} \bar y  + \frac{1}{b} x^2 + \mu \bar x y,
%\end{array}
\label{eq:Teps}
\end{equation}
%$$
where $\mu$ is a small parameter. This family is a representative of the class (\ref{eq:H12-1e}) of reversible perturbations given by Lemma~\ref{lm:H12tld}, and, thus, it preserves the reversibility with respect to the involution $h: x\to y, y\to x$.

In Figure~\ref{fig-bdiag8} the main elements of bifurcation diagrams for fixed points of the maps $T_2$ and $T_{2\mu}$ are represented in the $(b,M)$-parameter plane for $\mu=0$ in Figure~\ref{fig-bdiag8}(a) and for a sufficiently small fixed $\mu$ in Figure~\ref{fig-bdiag8}(b). We exclude a small strip containing the axis $b=0$ from the consideration since the maps $T_2$ and $T_{2\mu}$ are not defined for $b=0$. The main bifurcation curves are the following: the fold bifurcation curves $F_1$ and $F_2$, the reversible pitchfork bifurcation curves $PF_1$ and $PF_2$ as well as several period-doubling curves $PD$ that are shown as gray dashed lines.
The equations of the curves are as follows:
$$
\begin{array}{rl}
F_1: & 4 (1+b\mu) M= - (b-1)^2, \text{ where } b<0, \\
F_2: & 4 (1+b\mu) M= - (b-1)^2, \text{ where }  b>0, \\
PF_1: & 4 M = ({3}+b\mu) (b-1)^2, \text{ where } b<0 \\
PF_2: & 4 M = ({3}+b\mu) (b-1)^2, \text{ where } b>0.
\end{array}
$$

In the conservative case $\mu=0$, the curves $F_1$ and $F_2$ correspond to the creation of fixed points of $T_2$. There appears a symmetric parabolic fixed point which is nondegenerate for all parameter values in $F_1$ and $F_2$, except for the point $Q^*(b=1,M=0)\in F_2$. The parabolic point bifurcates into 2 symmetric elliptic  and saddle fixed points. The transition through the point $Q^*(b=1,M=0)$ corresponds to a codimension 2 bifurcation which consists in the emergence of 4 fixed points: 2 symmetric elliptic  and 2 asymmetric saddle fixed points which compose a symmetric couple of points.

\begin{figure}[t]
	\centerline{\epsfig{file=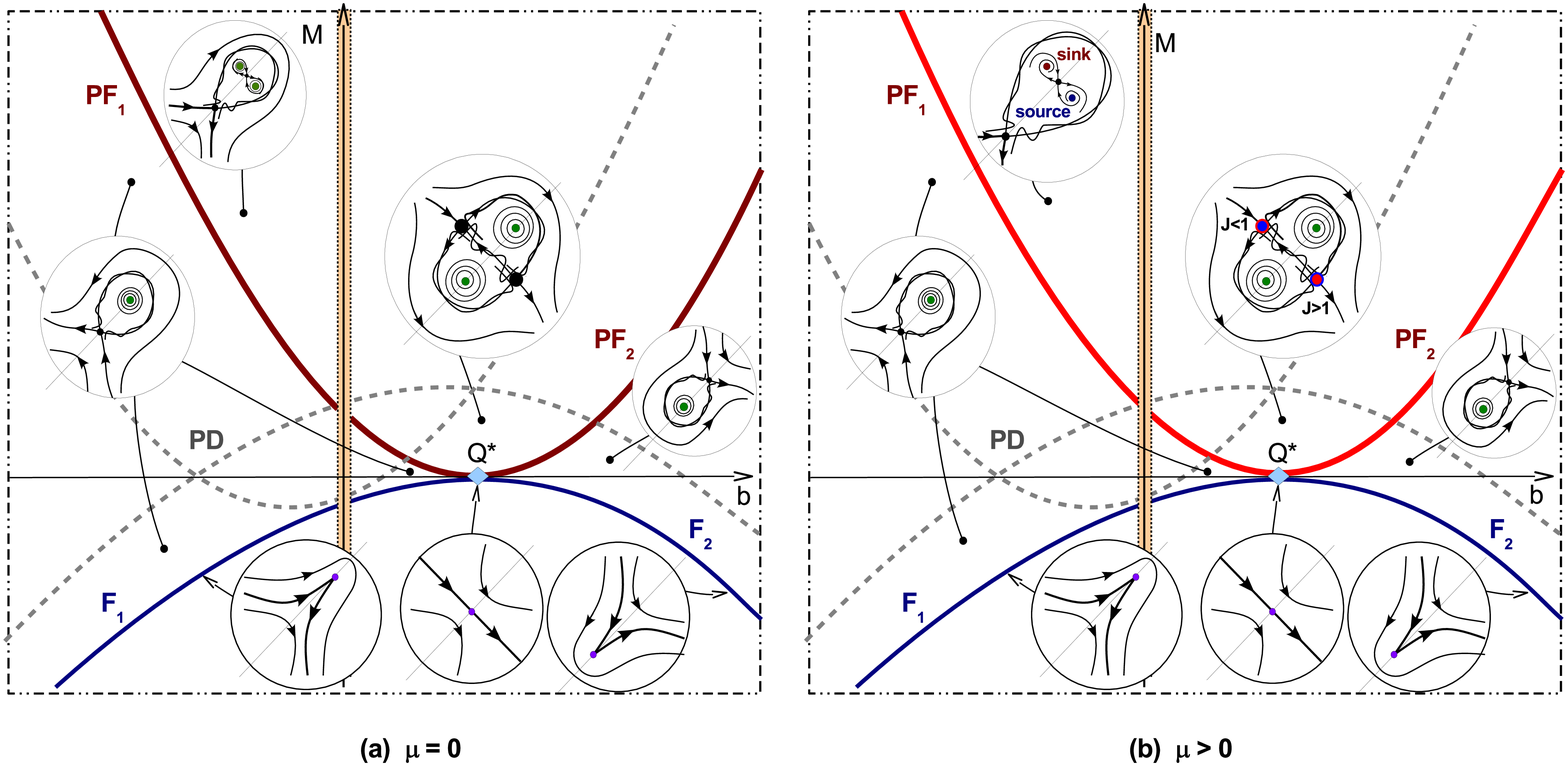, width=18cm}}
	\caption{Elements of the bifurcation diagram in the $(b,M)$-parameter plane for the maps (a) $T_2$ and (b) $T_{2\mu}$  with small fixed $\mu$.}
	\label{fig-bdiag8}
\end{figure}

In the perturbed case $\mu \neq 0$, the character of the fold and period-doubling bifurcations is not changed qualitatively if $\mu$ is sufficiently small.
However,  the pitchfork bifurcations can give rise to non-conservative fixed points. %If in~(\ref{Teps}) $\varepsilon(x,y) = \mu xy$,
It follows from Lemma~\ref{lm:H12tld} that the Jacobian of~(\ref{eq:Teps}) is
\begin{equation}
J= \frac{1+ b\mu\bar y}{1+ b\mu x}
\label{eq:JTepsfp}
\end{equation}

In order to find the fixed points of~(\ref{eq:Teps}), we equate $\bar x=x, \bar y =y$ and obtain the following system
%$$
\begin{equation}
\begin{array}{l}
x(b-1) = y^2 - M + b\mu xy, \;\;\; y(b-1) = x^2 - M + b\mu xy.
\end{array}
\label{eq:Tepsfp}
\end{equation}
Subtracting and adding up the equations and taking into account that $x\neq y$ (for asymmetric fixed points) gives us
$$
x+y = 1-b, \;\;\;
xy = \frac {(1-b)^2 -M}{1-b\mu}.
$$
Thus, if
$$
D = \frac{ 4 M - (1-b)^2(3+b\mu)}{4(1-b\mu)}>0,
$$
then two asymmetric fixed points $M_1$ and $M_2$ appear for $T_{2\mu}$:
$$
M_1=\left(\frac{1-b}{2} + \sqrt{D}, \frac{1-b}{2} - \sqrt{D}\right)\;\;\mbox{and}\;\; M_2=\left(\frac{1-b}{2} - \sqrt{D}, \frac{1-b}{2} + \sqrt{D}\right).
$$
These two points are symmetric with respect to the line $y=x$ and they merge  with the corresponding symmetric fixed point under a reversible pitchfork bifurcation (which occurs when $D=0$).

It follows from~(\ref{eq:JTepsfp}) that the Jacobian at the fixed points $M_1$ and $M_2$ is
$$
J_{1} = 1 - \frac{4b\mu\sqrt{D}}{2 + b\mu(1-b + 2\sqrt{D})}, \;\;\;
J_{2} = 1 +  \frac{4b\mu\sqrt{D}}{2 + b\mu(1-b - 2\sqrt{D})},
$$
respectively. Thus, if $b\mu>0$,  then $J_1<1$ and $J_2 = J_1^{-1}> 1$.

The topological type of these points (for small $\mu$) is easily determined from the conservative approximation $\mu=0$, see Figure~\ref{fig-bdiag8}a. In the case $b<0$, the points $M_1$ and $M_2$ compose a symmetric couple of elliptic fixed points for $\mu=0$. For $\mu\neq 0$, they are transformed into a symmetric couple of ``sink-source'' fixed points: the points  $M_1$ and $M_2$ become stable and unstable foci, respectively, if $\mu>0$.
In the case $b>0$, the points  $M_1$ and $M_2$ become non-conservative saddles for $\mu\neq 0$: with the Jacobians $J_1<1$ and $J_2>1$, respectively, if $\mu>0$.

We also note that in the case $|b|=1$, the map $T_{2\mu}$ of the form (\ref{eq:Teps}) gives an example of a reversible perturbation for the second iteration of the conservative H\'enon map, the orientable one at $b=-1$ and nonorientable at $b=+1$. However, these perturbations are not suitable for the H\'enon maps themselves. Thus, at $b=-1$, the curve $PF_1$ is, in fact, the period-doubling curve for a symmetric fixed point which means that proper reversible perturbations can not lead to symmetry breaking. In the next sections, we consider the questions on correct reversible perturbations for the H\'enon maps, nonorientable and orientable, and on the structure of the accompanying symmetry breaking bifurcations.

\subsection{Symmetry breaking bifurcations in the nonorientable reversible H\'enon maps. } \label{sec:bifnor}

As an example we consider now the nonorientable H\'enon map $H_{-1}$ of the form %~(\ref{eq:HM-})
\begin{equation}
\bar x = -y, \bar y = -M -x + y^2 ,
\label{eq:HM-}
\end{equation}
that is a particular case of the map~(\ref{eq:TH-0}).

\begin{figure} [t]
	\centerline{\epsfig{file=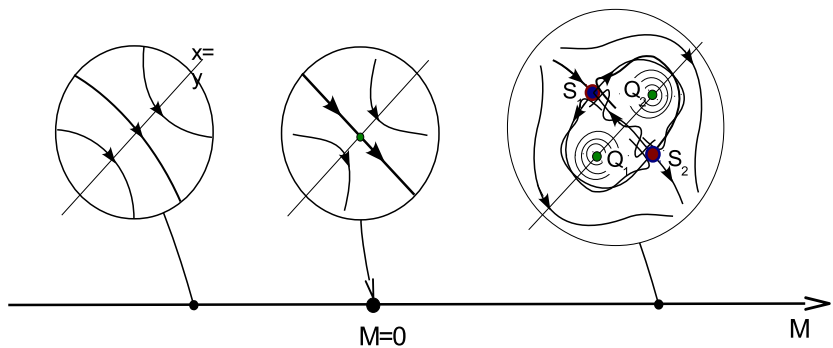, width=12cm}}
	\caption{{\footnotesize Main bifurcations in maps (\ref{eq:HM-}) and (\ref{eq:Hen-mu}) when $\mu$ is small and fixed and $M$ changes.
			The points $Q_1$ and $Q_2$ are symmetric elliptic 2-periodic orbits, while the points $S_1$ and $S_2$ are nonorientable saddle fixed points that compose a symmetric couple of points. The points $S_1$ and $S_2$ are conservative with the Jacobian $-1$ for $\mu=0$ and non-conservative with the Jacobians  $J_1<-1$ and $-1<J_2=J_1^{-1}<0$, respectively, for $\mu>0$.} }
	\label{fig-bdiagH2}
\end{figure}

For $M<0$, the map~(\ref{eq:HM-})  has no fixed points or periodic orbits. However, they appear immediately for $M>0$ under the so-called fold-flip bifurcation occurring  at $M=0$ when the map has a fixed point $P(0,0)$ with eigenvalues $+1$ and $-1$. For $M>0$, this point splits into 4 points, see Figure~\ref{fig-bdiagH2}, two of them are the fixed  points $S_1$ and $S_2$, and the other two points form a 2-periodic orbit $(Q_1,Q_2)$, i.e. $H_{-1}(Q_1)=Q_2, \; H_{-1}(Q_2)=Q_1$. Note that points $Q_1$ and $Q_2$ are elliptic 2-periodic orbits, and they are also symmetric since they both belong to the symmetry line $x=y$. In contrast, the points $S_1$ and $S_2$ are saddles and they compose a symmetric couple of points, i.e. $h(S_1)=S_2$ and $h(S_2)=S_1$. The coordinates of these points are
$$
Q_1 = (-\sqrt{M},-\sqrt{M}),\; Q_2 = (\sqrt{M},\sqrt{M}),\; S_1 = (-\sqrt{M},\sqrt{M}),\; S_2 = (\sqrt{M},-\sqrt{M}).
$$
All these points are conservative points of the map~(\ref{eq:HM-}).

However, due to Lemma~\ref{lm:Hninveps}, adding reversible perturbations  we can destroy the conservativity of the fixed points $S_1$ and $S_2$.
For example, let us consider the following perturbed map
\begin{equation}
\tilde{H}_{-1\mu}: \;\;\; \bar x = -y, \;\;\; \bar y + \mu \bar x \bar y = -M -x + y^2 - \mu xy,
\label{eq:Hen-mu}
\end{equation}
where we have chosen the perturbation $\varepsilon(x,y)=\mu x y$, being $\mu$ a small parameter.
By Lemma~\ref{lm:Hninveps}, this map is reversible with respect to the involution $h$, however, it is no longer conservative for $\mu \neq 0$. Indeed, formula~(\ref{eq:J-eps0}) reads as
$$
J = - \frac{1+\mu y}{1+\mu \bar x} = - \frac{1+\mu y}{1-\mu y}.
$$
The fixed points of map (\ref{eq:Hen-mu}) are easily found:
$S_1 = (-a(\mu), a(\mu)),\; S_2 = (a(\mu), -a(\mu))$, where $
a(\mu) = \sqrt{{M}/{(1+2\mu)}}.
$
Then
we have that the Jacobian at the points $S_1$ and $S_2$ are
$$
J_1= - 1 - \frac{2\mu \sqrt{M}}{\sqrt{1+2\mu} - \mu \sqrt{M}},\;\;\; J_2= - 1 + \frac{2\mu \sqrt{M}}{\sqrt{1+2\mu} + \mu \sqrt{M}},
$$
respectively.
Thus, if  $M>0$ and $\mu >0$ is not very large, the points $S_1$ and $S_2$ compose a symmetric couple of (nonorientable) saddles with the Jacobians $J_1<-1$ and $-1<J_2=J_1^{-1}<0$.

\subsection{Symmetry breaking bifurcations in the orientable reversible H\'enon maps.} \label{sec:sborH}

As an example we consider the standard area-preserving and orientable H\'enon map $H_{+1}$ of the form %~(\ref{eq:HM+}).
\begin{equation}
\bar x = y, \;\; \bar y = M - x - y^2, \;\;  .
\label{eq:HM+ }
\end{equation}
Bifurcations of its fixed points are well-known and include a parabolic bifurcation at $M=-1$, giving rise to symmetric elliptic and saddle fixed points, and a conservative period-doubling bifurcation of the fixed elliptic point at $M=3$, after which the elliptic fixed point becomes saddle and an elliptic 2-periodic orbits are born. Besides, when $M$ changes from $M=-1$ to $M=3$, the symmetric elliptic fixed point undergoes infinitely many bifurcations related to the appearance of resonant periodic orbits of period $q$ in its neighbourhood -- whenever the eigenvalues $e^{\pm i\varphi}$ pass through the values $\varphi = 2\pi \frac{p}{q}$, where $p$ and $q$ are mutually prime natural numbers and $p<q$.

However, as it is well-known, the fixed points and 2-periodic orbits in the H\'enon map are symmetric. The resonant periodic points are also symmetric if the resonances are nondegenerate. Thus, 3-periodic and 5-periodic resonant orbits are symmetric.
Although, the 1:4 resonance (related to eigenvalues $e^{\pm i\pi/2}=\pm i$) is degenerate in the H\'enon map \cite{Bir87, SimoVieiro09} (the so-called Arnold degeneracy \cite{Arn} takes place here), this bifurcation is not of symmetry breaking type \cite{LT12}.

A simple calculation of the number of points in periodic orbits of periods 1, 2, 3 and 4 (this number cannot be greater than $2^n$, by Bezout's theorem, even if we
include all points of periodic orbits of periods divisors of $n$)
\begin{itemize}
\item
period 1 (fixed points) -- two points;
\item
period 2 -- two points (we exclude 2 fixed points) that compose one 2-periodic orbit appearing after the period-doubling bifurcation of the elliptic fixed point;
\item
period 3 -- 6 points (we exclude 2 fixed points and the remaining 6 points compose two   (elliptic and saddle) 3-periodic orbits accompanying the 1:3 resonance);
\item
period 4 -- 12 points (we exclude 2 fixed points and 2 points of the 2-periodic orbits and, thus, the remaining 12 points form two 4-periodic orbits born from the 1:4 resonance and one 4-periodic orbit appearing after a period-doubling bifurcation of the  elliptic 2-periodic orbit);
\end{itemize}
shows that there are no asymmetric periodic orbits of these periods. 

The case of period 5 points is more delicate. If two fixed points are excluded, then 30 more points remain. 20 such points compose four 5-periodic orbits born from the 1:5 and 2:5 resonances. Concerning remaining 10 points, they appear at a symmetric parabolic bifurcation of 5-periodic orbit. The last bifurcation we have found numerically at $M=5.5517$. The $y$-coordinates of the corresponding symmetric parabolic point is $y_1=y_2=-2.243751084, y_3=y_5=2.761032157, y_4=0.172152512$.

The calculation of the number of points of 6-periodic orbits shows the following. There are 64 such points in total. They include 10 points of smaller periods: 2 fixed points, 2 points of the 2-periodic orbit and 6 points of the 3-periodic orbits. The remaining 54 points form 9  orbits of period 6. Among them, 5 orbits are symmetric -- two 6-periodic orbits are born from the 1:6 resonance, one periodic orbit appears via the period-doubling bifurcation of the elliptic 3-periodic orbit, and two orbits arise due to the 1:3 resonance of a 2-periodic elliptic orbit.

The remaining 4  orbits of period 6 may be asymmetric.
For example, some of these orbits can appear as a result of a symmetry breaking bifurcation when a symmetric couple of two parabolic 6-periodic orbits appears and then splits into
two symmetric couples of elliptic and saddle 6-periodic orbits. Other possible cases can be related to two bifurcations of symmetric 6-periodic orbits and at least one of these bifurcations is a pitchfork bifurcation. We show below, see Section~\ref{sec:FO6}, that the second possibility is indeed realized in the H\'enon map  $H_{+1}$.

We found numerically one couple of such orbits $O_6^1$ and $O_6^2$. In particular, for $M=4$, the orbit $O_6^1 =\{(x_i,y_i)\}$, $i=1,...,6$, where
$x_{i+1} = y_i$, has the following $y$-coordinates
$$
\begin{array}{l}
y_1=2.114907541,\;\;\; y_2=-1.935432332, \;\;\; y_3=-1.860805853, \\
y_4=2.472833909, \;\;\; y_5=-0.254101688, \; \;\; y_6=1.462598423.
\end{array}
$$
The orbit $O_6^2$ is symmetric to $O_6^1$ and, thus, has coordinates $\tilde x_i = y_i$ and $\tilde y_i = \tilde x_{i+1}$.

Now we consider the reversible perturbation of the H\'enon map  $H_{+1}$ as follows
\begin{equation}
\tilde{H}_{+1\mu}: \;\;\; \bar x = y, \bar y  + \mu (\bar x\bar y + \bar y^2)= M -x - y^2 - \mu  (xy + x^2),
\label{eq:revHen+}
\end{equation}
that preserve reversibility of the H\'enon map due to Lemma~\ref{lm:til_H}, see also Example~1 for $\varepsilon_1(x,y) = \mu (xy + y^2)$.

We note that in the perturbed map (\ref{eq:revHen+}) the orbit $O_6^1$ has at $\mu = 0.01$ the following $y_i$-coordinates (here again $x_{i+1}=y_i$)
$$
\begin{array}{l}
y_1=2.107429699, \;\;\; y_2=-1.911473368, \;\;\; y_3=-1.833980679, \\
y_4=2.460965013, \;\;\; y_5=-0.2062196180, \;\;\; y_6=1.423687035.
\end{array}
$$

We calculate the Jacobian of map~(\ref{eq:revHen+}) at $O_6^1$
$$
J= \prod_{i=1}^{6} \frac{1+\mu y_i + 2\mu y_{i-1}}{1 + \mu y_i + 2\mu y_{i+1}}
$$
and obtain that $J= 0.9999999555$.

\subsubsection{Search of the asymmetric 6-periodic orbit} \label{sec:FO6}

It is very surprising that the main bifurcations related to the appearance of 6-periodic orbits $O_6^1$ and $O_6^2$ can be studied analytically due to the fact that these orbits are born as a result of a symmetry breaking bifurcation of a symmetric 6-periodic orbit.

The corresponding bifurcation scenario starts at the value $M=M_1=\frac{5}{4}$ when an elliptic  3-periodic orbit $O_3$ undergoes a supercritical period-doubling bifurcation after which the orbit $O_3$ becomes a symmetric saddle 3-periodic orbit and  a symmetric elliptic 6-periodic orbit $\tilde O_6$ is born in its neighbourhood. Then increasing $M$, two successive period-doubling bifurcations  of the orbit $\tilde O_6$, supercritical (at $M=M_2\approx 1.2813$) and subcritical (at $M=M_3\approx 2.98038$), take place. For $M_2<M<M_3$ the orbit $\tilde O_6$ is saddle, and for $M>M_2$ it becomes elliptic again. An important bifurcation occurs at $M=M_4=3$ when the orbit $\tilde O_6$ undergoes a pitchfork bifurcation after which the orbit $\tilde O_6$ becomes a symmetric saddle 6-periodic orbit and a symmetric couple of elliptic 6-periodic orbits $O_6^1$ and $O_6^2$ is born, see Figure~\ref{Henon2D_PF}.

\begin{figure} [t]
	\centerline{\epsfig{file=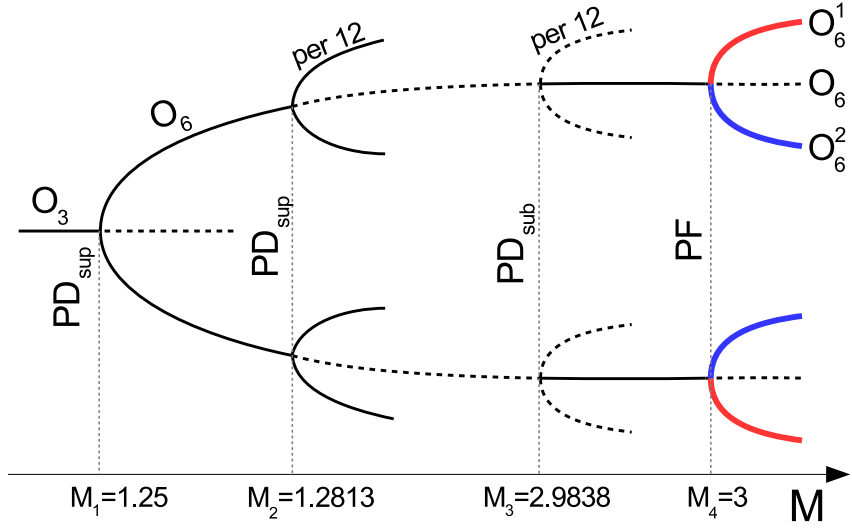, width=12cm}}
	\caption{{\footnotesize A schematic tree for the bifurcation scenario of the appearance of a symmetric couple of 6-periodic orbits in the H\'enon map $H_{+1\mu}$, $\mu=0.01$. } }
	\label{Henon2D_PF}
\end{figure}

Let us explain this scenario in more detail. The 3-periodic orbits appear in the H\'enon map $H_{+1}$ at $M=1$ under a symmetric parabolic bifurcation leading to the birth of elliptic and saddle orbits $O_3$ and $S_3$. At $M=M_1=\frac{5}{4}$ the saddle orbit $S_3$ merges with the fixed point  $O_{1:3}(\frac{1}{2},\frac{1}{2})$ with eigenvalues $e^{\pm 2\pi/3}$ (the 1:3 resonance) and simultaneously (exactly at this moment $M=M_1=\frac{5}{4}$) the elliptic orbit $O_3$ undergoes a supercritical period-doubling bifurcation giving rise to  elliptic and saddle 6-periodic orbits, see Figure~\ref{Henon2D_PF_portr} (a)-(d).  The elliptic orbit is the orbit $\tilde O_6$. As this orbit is symmetric, it has two intersection points with the line $x=y$. Let $\tilde O_6 = \{P_i(x_i,y_i)\},\; i=1,2,...,6$. Then the coordinates $(x_i,y_i)$ satisfy the following equations
$$
x_{i+1} = y_i,\; y_{i+1} = M - x_i - y_i^2, \;\; %i = i (\mbox{mod}6)1,..., 6
$$
where $i = 1,..., 6$. % and we put $x_7=x_1,y_7 = y_1$.
Since $x_{i+1} = y_i$, we can reduce this system to the following system of 6 quadratic equations
\begin{equation}
y_{2} = M - y_{6} - y_1^2, \; y_{3} = M - y_{1} - y_2^2, \; y_{4} = M - y_{2} - y_3^2, \\
y_{5} = M - y_{3} - y_4^2, \; y_{6} = M - y_{4} - y_5^2, \; y_{1} = M - y_{5} - y_6^2 .
% i = 1,..., 6
\label{eq:6eqy}
\end{equation}

\begin{figure} [t]
	\centerline{\epsfig{file=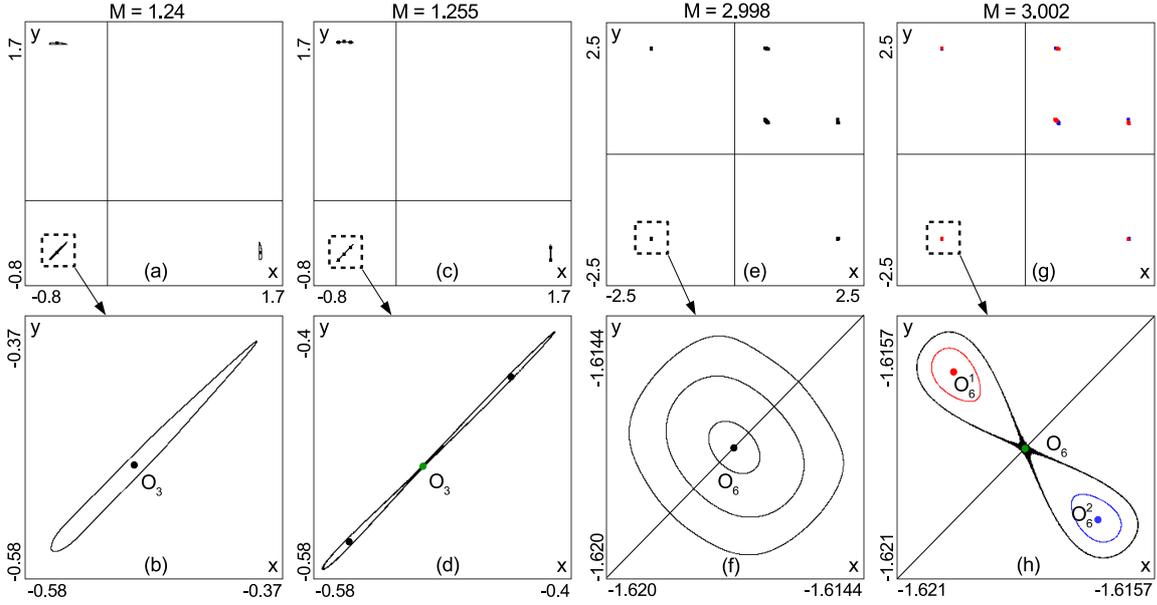, width=16cm}}
	\caption{{\footnotesize Phase portraits of 3- and 6-periodic orbitsfor the H\'enon map $H_{+1}$: the bottom plots are magnifications of some important details of the top plots.    } }
	\label{Henon2D_PF_portr}
\end{figure}

Assume that the point $P_1(x_1,y_1)$ of $\tilde O_6$ belongs to the symmetry line $x=y$, i.e. $x_1=y_1$. Then the point $P_4=H_{+1}^3(P_1)$ is also symmetric, i.e. $x_4=y_4$. Since $x_1=y_6$ and $x_4 =y_3$ and, hence,
$y_6=y_1$ and $y_3=y_4$. Then we get from the first and the last equations of (\ref{eq:6eqy}) that $y_2 = y_5$ and, thus, the system (\ref{eq:6eqy}) is reduced to the following system of three equations
\begin{equation}
y_{1} + y_2 = M - y_1^2, \; y_{1} + y_3 = M - y_2^2, \; y_{2} + y_3  = M - y_3^2, \\
\label{eq:3eqy}
\end{equation}
From the first and last equations of (\ref{eq:3eqy}) we obtain the relation
$y_1-y_3 = (y_3-y_1)(y_3 + y_1)$. If $y_1=y_3$, then the corresponding orbit has period 3. It follows that $y_1+y_3 = -1$. The second equation of (\ref{eq:3eqy}) gives us that $y_2^2 = M+1$. Then $y_1$ and $y_3$ satisfy the relations
$$
y_1^2 +y_{1} - M \pm \sqrt{M+1} =0, \; y_3^2 +y_{3} - M \pm \sqrt{M+1} =0.
$$
These two equations have the same solution, but $y_1$ and $y_3$ should take different values. Then, assuming for more definiteness that $y_3>y_1$, we find the following $y_i$-coordinates for the two symmetric 6-periodic orbits $\tilde O_6^1$ and $\tilde O_6^2$: for $\tilde O_6^1$
$$
y_1=y_6 = \frac{1}{2}\left(-1-\sqrt{1-4\sqrt{M+1}+4M} \right), y_2=y_5 =\sqrt{M+1}, y_3=y_4 = \frac{1}{2}\left(-1+\sqrt{1-4\sqrt{M+1}+4M} \right);
$$
for $\tilde O_6^2$
$$
y_1=y_6 = \frac{1}{2}\left(-1-\sqrt{1+4\sqrt{M+1}+4M} \right), y_2=y_5 =-\sqrt{M+1}, y_3=y_4 = \frac{1}{2}\left(-1+\sqrt{1+4\sqrt{M+1}+4M} \right).
$$

We stress that the orbit $\tilde O_6^1$ is born at $M = \frac{5}{4}$ (when $1-4\sqrt{M+1}+4M = 0$) under a supercritical period-doubling bifurcation of the elliptic 3-periodic orbit, while the orbit $\tilde O_6^2$ appears at $M = \frac{-3}{4}$ (when $1 + 4\sqrt{M+1}+4M = 0$) via a bifurcation of the 1:6 resonance fixed point.

Further we consider only the orbit $\tilde O_6^1$. In the analysis of its bifurcations we find the trace $Tr$ of the characteristic matrix for the map $H_{+1}^6$ at some point of $\tilde O_6^1$. As a result, we obtain that
$$
Tr = 86+24\sqrt{M+1}+116M-128\sqrt{M+1}M+96M^2-128\sqrt{M+1}M^2+64M^3.
$$
If we denote $\sqrt{M+1}=x$ ($x>0$), we obtain the polynomial
$$
Tr(x) = 2 + 24x + 116 x^2 + 128 x^3 -96x^4 -128x^5 + 64 x^6 = 2+ 4 x (2 x + 1)^3 (2 x - 3) (x - 2)
$$
Thus, the equation $Tr(x)=2$ has the solutions $x=0$, $x=-\frac{1}{2}$ (the triple root) and two positive solutions $x =\frac{3}{2}$ and $x=2$. The root $x=\frac{3}{2}$ corresponds to the value $M = \frac{5}{4}$ when the orbit $\tilde O_6^1$ is born.
The root $x=2$ corresponds to the value $M = 3$ when the symmetric elliptic orbit $\tilde O_6^1$ undergoes a pitchfork bifurcation -- the elliptic orbit becomes symmetric saddle and a symmetric couple of elliptic 6-periodic orbits $O_6^1$ and $O_6^2$ emerges, see Figure~\ref{Henon2D_PF_portr}(e)-(h). Namely, the orbits $O_6^1$ and $O_6^2$ are considered in Section~\ref{sec:sborH}. \\~\\

\textbf{Acknowledgements.} The authors thank D.V. Turaev and A.O. Kazakov for very useful remarks. This work is supported by the Russian Science Foundation under grants 19-11-00280 (Sections 1, 2 and 3) and 19-71-10048 (subsection 4.3). The authors thank also the Russian Foundation for Basic Research, project nos. 19-01-00607 (subsection 3.3) and 18-29-10081 (subsection 2.5) for support of scientific researches.
Numerical experiments in Section 4 were supported by the Laboratory of Dynamical Systems and Applications NRU HSE, of the Russian Ministry of Science and Higher Education (Grant No. 075-15-2019-1931). S.G. and K.S. acknowledge the financial support of the Ministry of Science and Higher Education of Russian Federation (Project \#0729-2020-0036).
M.G.  was partially supported by the Spanish grants Juan de la Cierva-Incorporaci\'on IJCI-2016-29071, PGC2018-098676-B-I00 (AEI/FEDER/UE) and the Catalan grant 2017SGR1374.

\end{document}